\documentclass{amsart}

\usepackage{amsmath,amssymb,amsfonts}
\newtheorem{theorem}{Theorem}[section]
\newtheorem{lemma}[theorem]{Lemma}

\theoremstyle{definition}
\newtheorem{definition}[theorem]{Definition}

\newtheorem{coro}[theorem]{Corollary}
\newtheorem{conj}[theorem]{Conjecture}

\theoremstyle{remark}

\numberwithin{equation}{section}

\def \x{\mathbf{x}}
\def \v{\mathbf{v}}
\def \u{\mathbf{u}}
\def \i{\mathbf{i}}
\def \A{\mathcal{A}}

\def \la{\lambda}
\def \I{\mathcal{I}}

\def \S{\mbox{Spec}}

\def\Z{\mathbb{Z}}

\begin{document}
\title[Cyclic index of generalized power hypergraphs]{The cyclic index of adjacency tensor of generalized power hypergraphs }

\author[Y.-Z. Fan]{Yi-Zheng Fan}
\address{School of Mathematical Sciences, Anhui University, Hefei 230601, P. R. China}
\email{fanyz@ahu.edu.cn}
\thanks{The first author is the corresponding author, and was supported by National Natural Science Foundation of China \#11871073.
}

\author[M. Li]{Min Li}
\address{School of Mathematical Sciences, Anhui University, Hefei 230601, P. R. China}
\email{1736808193@qq.com}





\begin{abstract}
Let $G$ be a $t$-uniform hypergraph, and let $c(G)$ denote the cyclic index of the adjacency tensor of $G$.
Let $m,s,t$ be positive integers such that $t \ge 2$, $s \ge 2$ and $m=st$.
The generalized power $G^{m,s}$ of $G$ is obtained from $G$ by blowing up each vertex into an $s$-set and preserving the adjacency relation.
It was conjectured that $c(G^{m,s})=s \cdot c(G)$.
In this paper we show that the conjecture is false by giving a counterexample, and give some sufficient conditions for the conjecture holding.
Finally we give an equivalent characterization of the equality in the conjecture by using a matrix equation over $\Z_m$.

%
\end{abstract}



\maketitle

\section{Introduction}
A hypergraph $G=(V(G), E(G))$ consists of a set of vertices, say $V(G)=\{v_{1},v_{2},\cdots v_{n}\}$,
   and a set of edges, say $E(G)=\{e_{1}, e_{2}, \cdots e_{k}\}$, where $e_{j}\subseteq V(G)$ for $j \in [k]:=\{1,2,\ldots,k\}$.
   If $|e_{j}|=m$ for each $j \in [k]$, then $G$ is called an $m$-uniform hypergraph.
A {\it walk} $W$ in $G$ is a sequence of alternating vertices and edges: $v_{i_0},e_{i_1},v_{i_1},e_{i_2},\ldots,e_{i_l},v_{i_l}$,
    where $\{v_{i_j},v_{i_{j+1}}\}\subseteq e_{i_{j+1}}$ for $j=0,1,\ldots,l-1$.
The hypergraph $G$  is {\it connected} if every two vertices of $G$ are connected by a walk.
The {\it adjacency tensor} $\A(G)$ of the hypergraph $G$ is defined as $\mathcal{A}(G)=(a_{i_{1}i_{2}\ldots i_{k}})$ \cite{cooper}, an $m$-th order $n$-dimensional tensor, where
\[a_{i_{1}i_{2}\ldots i_{m}}=\left\{
 \begin{array}{ll}
\frac{1}{(m-1)!}, &  \mbox{if~} \{v_{i_{1}},v_{i_{2}},\ldots,v_{i_{m}}\} \in E(G);\\
  0, & \mbox{otherwise}.
  \end{array}\right.
\]

In general, A \emph{tensor} (also called \emph{hypermatrix}) $\A=(a_{i_{1} i_2 \ldots i_{m}})$ of order $m$ and dimension $n$ over a field $\mathbb{F}$ refers to a
 multiarray of entries $a_{i_{1}i_2\ldots i_{m}}\in \mathbb{F}$ for all $i_{j}\in [n]$ and $j\in [m]$,
  which can be viewed to be the coordinates of the classical tensor (as a multilinear function) under an orthonormal basis.
If $m=2$, then $\A$ is a square matrix of dimension $n$.

In 2005, independently, Lim \cite{Lim} and Qi \cite{Qi} introduced eigenvalues for tensors $\A$.
Denote by $\rho(\A)$ the spectral radius of $\A$, and by $\S(\A)$ the spectrum of $\A$.
If $\A$ is further nonnegative, then by Perron-Frobenius theorem of nonnegative tensors,
   $\rho(\A)$ is an eigenvalue of $\A$.
Moreover, if $\A$ is weakly irreducible and has $k$ eigenvalues of $\A$ with modulus $\rho(\A)$, then those $k$ eigenvalues are equally distributed on the spectral circle.
As for nonnegative matrices, the number $k$ is called the cyclic index of $\A$ \cite{CPZ2}.
The cyclic index reflects the spectral symmetry of nonnegative weakly irreducible tensors, which was generalized and investigated in the paper \cite{FHBZL}.

\begin{definition}[\cite{FHBZL}]\label{ell-sym}
Let $\A$ be an $m$-th order $n$-dimensional tensor, and let $\ell$ be a positive integer.
The tensor $\A$ is called {\it spectral $\ell$-symmetric} if
\begin{equation}\label{sym-For}\S(\A)=e^{\i \frac{2\pi}{\ell}}\S(\A).\end{equation}

The maximum number $\ell$ such that (\ref{sym-For}) holds is called the {\it cyclic index} of $\A$ and denoted by $c(\A)$, and $\A$ is called
{\it spectral $c(\A)$-cyclic}.
\end{definition}

When we say a hypergraph is {\it spectral $\ell$-symmetric} or {\it spectral $\ell$-cyclic},
this is always referring to its adjacency tensor.
In particular, for a uniform hypergraph $G$, denote $c(G):=c(\A(G))$, called the cyclic index of $G$.

For a general tensor $\A$, if it is spectral $\ell$-symmetric, then $\ell | c(\A)$ by \cite[Lemma 2.7]{FHBZL}.
It was also proved that if a connected $m$-uniform hypergraph is spectral $\ell$-symmetric, then $\ell |m$,
and hence $c(G) |m$; see \cite[Lemma 3.2, Corollary 4.3]{FHBZL}, \cite[Lemma 2.5]{FBH} or \cite[Theorem 2.15]{YY2}.
In the paper \cite{FHBZL} the authors use the construction of generalized power hypergraphs to show that
for every positive integer $m$ and any positive integer $\ell$ such that $\ell | m$, there always exists an $m$-uniform hypergraph $G$ such that $G$ is spectral $\ell$-symmetric.
They posed the following conjecture.

\begin{conj}[\cite{FHBZL}]\label{conj} let $G$ be a $t$-uniform hypergraph, and let $G^{m,s}$ be the generalized power of $G$, where $m=st$. Then
\begin{equation}\label{conjFm} c(G^{m,s})=s \cdot c(G). \end{equation}
\end{conj}

The generalized power of a hypergraph is defined as follows.

\begin{definition}[\cite{KLQY16}] \label{powerhyper}
Let $G=(V,E)$ be a $t$-uniform hypergraph. For any integers $m,s$ such that $m > t$ and $1 \le s \le \frac{m}{t}$, the generalized power of $G$, denoted by $G^{m,s}$, is defined as the $m$-uniform hypergraph with the vertex set $(\cup_{v \in V} \v) \cup (\cup_{e \in E}\mathbf{e})$, and the edge set
$\{\u_1 \cup \cdots \cup \u_t \cup \mathbf{e}: e=\{u_1,\ldots,u_t\} \in E(G)\}$,
where $\v$ denotes an $s$-set corresponding to $v$ and $\mathbf{e}$ denotes an $(m-ts)$-set corresponding to $e$, and all those sets are pairwise disjoint.
\end{definition}

In this paper, we only consider the power hypergraphs $G^{m,s}$ with $m=st$, i.e.
  $G^{m,s}$ is obtained from $G$ by blowing up each vertex into an $s$-set and preserving the adjacency relation.
The generalized power hypergraphs include some special cases, such as the powers of simple graphs introduced by Hu, Qi and Shao \cite{HQS13},
the generalized powers of simple graphs introduced by Khan and Fan \cite{KF15}.
Peng \cite{Peng} introduced $s$-paths and $s$-cycles with uniformity $m$ on discussing the Ramsey number,
which are exactly the generalized pows of paths and cycles (as simple graphs) respectively if $1 \le s \le \frac{m}{2}$.
The spectral results on generalized power hypergraphs can be found in \cite{HQS13,ZSWB14,KF15,YQS16,KFT16,KLQY16}.

%

For the conjecture \ref{conj}, it was shown that it is true if $c(G)=1$ \cite{FHBZL}.
In this paper we show that the conjecture is false by giving a counterexample, and give some sufficient conditions for the conjecture holding.
We finally give an equivalent characterization of Eq. (\ref{conjFm}) by using a matrix equation over $\Z_m$.



\section{Preliminaries}
\subsection{Notions}
Let $\A$ be a real tensor of order $m$ and dimension $n$.
The tensor $\A$ is called \textit{symmetric} if its entries are invariant under any permutation of their indices.
So, the adjacency tensor of a uniform hypergraph is symmetric.

 Given a vector $x\in \mathbb{C}^{n}$, $\A x^{m} \in \mathbb{C}$ and $\A x^{m-1} \in \mathbb{C}^n$, which are defined as follows:
   \[
   \begin{split} \A x^{m}&=\sum_{i_1,i_{2},\ldots,i_{m}\in [n]}a_{i_1i_{2}\ldots i_{m}}x_{i_1}x_{i_{2}}\cdots x_{i_m},\\
   (\A x^{m-1})_i &=\sum_{i_{2},\ldots,i_{m}\in [n]}a_{ii_{2}\ldots i_{m}}x_{i_{2}}\cdots x_{i_m}, i \in [n].
   \end{split}
   \]
 Let $\mathcal{I}=(i_{i_1i_2\ldots i_m})$ be the {\it identity tensor} of order $m$ and dimension $n$, that is, $i_{i_{1}i_2 \ldots i_{m}}=1$ if and only if
   $i_{1}=i_2=\cdots=i_{m} \in [n]$ and $i_{i_{1}i_2 \ldots i_{m}}=0$ otherwise.

\begin{definition}[\cite{Lim,Qi}]\label{eigen} Let $\A$ be an $m$-th order $n$-dimensional real tensor.
For some $\lambda \in \mathbb{C}$, if the polynomial system $(\lambda \mathcal{I}-\A)x^{m-1}=0$, or equivalently $\A x^{m-1}=\lambda x^{[m-1]}$, has a solution $x\in \mathbb{C}^{n}\backslash \{0\}$,
then $\lambda $ is called an eigenvalue of $\A$ and $x$ is an eigenvector of $\A$ associated with $\lambda$,
where $x^{[m-1]}:=(x_1^{m-1}, x_2^{m-1},\ldots,x_n^{m-1})$.
\end{definition}

The \emph{determinant} of $\A$, denoted by $\det \A$, is defined as the resultant of the polynomials $\A \x^{m-1}$ \cite{Ha},
and the \emph{characteristic polynomial} $\varphi_\A(\la)$ of $\A$ is defined as $\det(\la \I-\A)$ \cite{Qi,CPZ3}.
It is known that $\la$ is an eigenvalue of $\A$ if and only if it is a root of $\varphi_\A(\la)$.
The \emph{spectrum} of $\A$ is the multi-set of the roots of $\varphi_\A(\la)$.

The spectral symmetry of a connected hypergraph is closed related to a certain coloring of the hypergraph.



\begin{definition}[\cite{FHBZL}]\label{spe-ell-sym-graph}
Let $m \ge 2$ and $\ell \ge 2$ be integers such that $ \ell \mid  m$.
An $m$-uniform hypergraph $G$ on $n$ vertices is called $(m,\ell)$-colorable
if there exists a map $\phi:[n] \to [m]$ such that if $\{i_1,\ldots, i_m\} \in E(G)$, then
\begin{equation}\label{gen-col} \phi(i_1)+\cdots+\phi(i_m) \equiv \frac{m}{\ell} \mod m.\end{equation}
Such $\phi$ is called an $(m,\ell)$-coloring of $G$.
\end{definition}

If $m$ is even, an $m$-uniform hypergraph with an $(m,2)$-coloring was called {\it odd-colorable} by Nikiforov \cite{Ni}.

\begin{theorem} \label{sym-col-graph}\cite{FHBZL}
Let $G$ be a connected $m$-uniform hypergraph.
Then $G$ is spectral $\ell$-symmetric if and only if $G$ is $(m,\ell)$-colorable.
\end{theorem}

The \emph{edge-vertex incidence matrix} $B_G=(b_{ev})$ of an $m$-uniform hypergraph $G$ is defined by
\begin{align*}
b_{ev}=\begin{cases}
1, & \mbox{if~} \ v\in e \in E(G),\\
0, & \textrm{otherwise}.
\end{cases}
\end{align*}
We may view $B_G$ as one over $\Z_m$, where $\Z_m$ is the ring of integers modulo $m$.
Now Eq. (\ref{gen-col}) is equivalent to
\begin{equation}
B_G \phi =\frac{m}{\ell} \mathbf{1} \textrm{~over~} \Z_m,
\end{equation}
where $\phi=(\phi(1),\ldots,\phi(n))$ is considered as a column vector, and $\mathbf{1}$ is an all-ones vector of dimension $n$.
So, Theorem \ref{sym-col-graph} can be rewritten as follows.

\begin{coro} \label{sym-Zm}
Let $G$ be a connected $m$-uniform hypergraph.
Then $G$ is spectral $\ell$-symmetric if and only if the equation
\begin{equation}
B_G x =\frac{m}{\ell} \mathbf{1} \textrm{~over~} \Z_m
\end{equation} has a solution.
\end{coro}

In Corollary \ref{sym-Zm} and other places of the paper, the number of coordinates of $\mathbf{1}$ is implicated from context, which is equal to the number of vertices of the hypergraph under discussion.

\section{Cyclic index of generalized power hypergraphs}
Let $G$ be a $t$-uniform hypergraph, and let $G^{m,s}$ be a generalized power hypergraph of $G$, where $1 \le s \le \frac{m}{t}$.
If $m > st$, then each edge of $G$ contains a vertex of degree $1$, and hence $G$ is a $1$-hm bipartite hypergraph \cite{SQH}.
By \cite[Theorem 3.2]{SQH} or \cite[Theorem 4.5]{FHBZL}, $c(G^{m,s})=m$.

So, in the following, we always assume that $G$ is a connected $t$-uniform hypergraph, $m=st$,
namely, $G^{m,s}$ is considered to be obtained from $G$ by blowing each vertex $v$ into an $s$-set $\v$ and preserving the adjacency relation.
We also assume that the vertex $v$ is contained in $\v$ for each $v \in V(G)$.

\begin{lemma}\label{sym-G-power}
If $G$ is spectral $\ell$-symmetric, then $G^{m,s}$ is also spectral $\ell$-symmetric.
In particular, $G^{m,s}$ is spectral $c(G)$-symmetric and hence $c(G) | c(G^{m,s})$.
\end{lemma}

\begin{proof}
Suppose that $G$ is spectral $\ell$-symmetric.
By Corollary \ref{sym-Zm}, the equation $B_G x =\frac{t}{\ell} \mathbf{1}$ has a solution $\phi$ over $\Z_t$.
Now define a map $\Phi$ on $G^{m,s}$ such that $\Phi|_{\v}=\phi(v)$ for each vertex $v \in V(G)$.
Then
$$ B_{G^{m,s}} \Phi =s \cdot B_G \phi=\frac{st}{\ell} \mathbf{1} = \frac{m}{\ell} \mathbf{1} \textrm{~over~} \Z_m,$$
which implies that $G^{m,s}$ is spectral $\ell$-symmetric also by Corollary \ref{sym-Zm}.
\end{proof}

\begin{lemma}\label{sym-s}
$G^{m,s}$ is spectral $s$-symmetric.
\end{lemma}

\begin{proof}
For each vertex $v \in V(G)$, $v$ is blowing into an $s$-set $\v$ of vertices of $G^{m,s}$, and is assumed to be contained in $\v$.
Define a map $\Phi$ on $G^{m,s}$ such that $\Phi(v)=1$ and $\Phi|_{\v \backslash \{v\}}=0$ for each vertex $v \in V(G)$.
Then
$$ B_{G^{m,s}} \Phi =t \mathbf{1} = \frac{m}{s} \mathbf{1} \textrm{~over~} \Z_m,$$
which implies that $G^{m,s}$ is spectral $s$-symmetric by Corollary \ref{sym-Zm}.
\end{proof}

\begin{lemma}\label{sym-power-G}
If $G^{m,s}$ is spectral $s \cdot \ell'$-symmetric, then $G$ is spectral $\ell'$-symmetric.
\end{lemma}

\begin{proof}
By Corollary \ref{sym-Zm}, there exists a map $\Phi$ defined on $G^{m,s}$ such that
$$ B_{G^{m,s}} \Phi =\frac{m}{s \cdot \ell'} \mathbf{1}=\frac{t}{\ell'} \mathbf{1} \textrm{~over~} \Z_m.$$
Now define a map $\phi$ on $G$ such that $\phi(v)=\sum_{u \in \v} \Phi(u)$ for each $v \in V(G)$.
So we have
$$ B_{G} \phi =B_{G^{m,s}} \Phi= \frac{t}{\ell'} \mathbf{1} \textrm{~over~} \Z_m.$$
As $m$ is a multiple of $t$,
$$ B_{G} \phi = \frac{t}{\ell'} \mathbf{1} \textrm{~over~} \Z_t,$$
which implies that $G$ is spectral $\ell'$-symmetric by Corollary \ref{sym-Zm}.
\end{proof}

By Lemma \ref{sym-s}, we may assume $c(G^{m,s})=s \cdot \ell'$, where $\ell'$ is a positive integer.
By Lemma \ref{sym-power-G}, we know that $G$ is spectral $\ell'$-symmetric and hence $\ell'|c(G)$ by \cite[Lemma 2.7]{FHBZL}.
So we get the following result immediately.

\begin{coro}\label{divide}
$c(G^{m,s}) | s \cdot c(G)$.
\end{coro}

\begin{coro}\label{part-conj}
$G^{m,s}$ is spectral $\frac{s \cdot c(G)}{(s,c(G))}$-symmetric.
In particular, if $(s,c(G))=1$ or $(s,t)=1$, then $c(G^{m,s})= s \cdot c(G)$.
\end{coro}

\begin{proof}
By Lemma \ref{sym-G-power} and Lemma \ref{sym-s}, we know that $c(G) | c(G^{m,s})$ and $s | c(G^{m,s})$, implying that $\frac{s \cdot c(G)}{(s,c(G))}|c(G^{m,s})$.
So, $G^{m,s}$ is spectral $\frac{s \cdot c(G)}{(s,c(G))}$-symmetric.
As $c(G)|t$, if $(s,t)=1$, then $(s,c(G))=1$.
If $(s,c(G))=1$, then $ s \cdot c(G)| c(G^{m,s})$.
Then result follows by Corollary \ref{divide}.
\end{proof}

By Corollary \ref{part-conj}, Conjecture \ref{conj} holds in some special cases, including the case of $c(G)=1$.
However,  Conjecture \ref{conj} does not hold in general.
Now we give a counterexample to show the negative answer to the conjecture.

\begin{definition}[\cite{Ni}]
Let $n \ge 16k$ and let partition $[n]$ into three sets $A, B, C$ such that $|A| \ge 6k$, $|B| \ge 6k$ and $|C| \ge 4k$.
Define the four families of $4k$-subsets of $[n]$.
\begin{align*}
E_1&:= \{e: e \subset [n], |e \cap A|=2k, |e \cap C|=2k\}.\\
E_2&:=\{e: e \subset [n], |e \cap B|=2k, |e \cap C|=2k\}.\\
E_3&:=\{e: e \subset [n], |e \cap A|=k, |e \cap B|=3k\}.\\
E_4&:=\{e: e \subset [n], |e \cap A|=3k, |e \cap B|=k\}.
\end{align*}
Now define a $4k$-uniform hypergraph $G$ by setting $V(G)=[n]$ and $E(G)=E_1 \cup E_2 \cup E_3 \cup E_4$.
We call $G$ a {\it Nikiforov's hypergraph} as it is introduced by Nikiforov.
\end{definition}

Nikiforov \cite{Ni} showed that Nikiforov's hypergraphs $G$ are odd-colorable, or $(4k,2)$-colorable in terms our definition, by defining a function $\phi$ on $G$ such that
$\phi|_A=1$, $\phi|_B=4k-1$ and $\phi|_C=0$.
By Theorem \ref{sym-col-graph}, $G$ is spectral $2$-symmetric.

By the following result, if $G$ is a Nikiforov's hypergraph and $s$ is even, then
$$c(G^{m,s}) \ne s \cdot c(G).$$
So we give a negative answer to Conjecture \ref{conj}.

\begin{theorem}
Let $G$ be a $4k$-uniform Nikiforov's hypergraph.
Then the following results hold.

\begin{enumerate}
\item $c(G)=2$.

\item If $s$ is even, then $c(G^{m,s})=s$.

\end{enumerate}

\end{theorem}

\begin{proof}
(1) We first show that $c(G)=2$.
Suppose that $G$ is spectral $\ell$-symmetric.
Then there exists a $\phi: [n] \to [4k]$ such that $B_G \phi =\frac{4k}{\ell}$  over $\Z_{4k}$.
It is easily seen that $\phi$ is constant on each of $A, B, C$ by the equation.
So, let $\phi|_A:=a$, $\phi|_B:=b$ and $\phi|_C:=c$.
Then, by considering the edges in $E_1$, we have
$$  2k a + 2k c = \frac{4k}{\ell} \mod 4k,$$
which implies that $\ell$ equals $1$ or $2$, and hence $c(G)=2$ as $G$ is spectral $2$-symmetric.

(2) By Corollary \ref{divide}, $c(G^{m,s}) | 2s$, where $m=4ks$.
By Lemma \ref{sym-s}, $G^{m,s}$ is spectral $s$-symmetric, and hence $s | c(G^{m,s})$.
We will show that if $s$ is even, then $G^{m,s}$ is not spectral $2s$-symmetric so that $c(G^{m,s})=s$.

Assume to the contrary that $G^{m,s}$ is spectral $2s$-symmetric.
Then there exists a $\Phi: V(G^{m,s}) \to [4ks]$ such that
$$B_{G^{m,s}} \Phi =\frac{4ks}{2s}=2k  \textrm{~over~} \Z_{4ks}.$$
For each $v \in V(G)$, define $\phi(v):=\sum_{u \in \v} \Phi(u)$.
So we have
$$B_{G^{m,s}} \Phi = B_G \phi=2k  \textrm{~over~} \Z_{4ks}.$$
It is also easily seen that $\phi|_A:=\alpha$, $\phi|_B:=\beta$ and $\phi|_C:=\iota$.
By considering the edges in $E_3$ and $E_4$ respectively, we have
$$ \alpha + 3 \beta = 2 \mod 4s, ~~~ 3 \alpha+ \beta = 2 \mod 4s.$$
So
$$ \alpha-\beta= 0 \mod 2s, ~~~ \alpha + \beta = 1 \mod s,$$
which yields a contradiction as $s$ is an even number.
\end{proof}

Finally we give an equivalent characterization of Eq. (\ref{conjFm}) in Conjecture \ref{conj}.

\begin{theorem}
$c(G^{m,s})= s \cdot c(G)$ if and only if the equation
\begin{equation}\label{eq-power} B_G x = \frac{t}{c(G)} \mathbf{1} \textrm{~over~} \Z_m \end{equation}
has a solution.
\end{theorem}

\begin{proof}
Suppose that $c(G^{m,s})= s \cdot c(G)$.
Then $G^{m,s}$ is spectral $s \cdot c(G)$-symmetric, and by Corollary \ref{sym-Zm},
there exists a map $\Phi: V(G^{m,s}) \to [m]$ such that
$$B_{G^{m,s}} \Phi =\frac{m}{s \cdot c(G)}\mathbf{1}=\frac{t}{c(G)}\mathbf{1}  \textrm{~over~} \Z_{m}.$$
For each $v \in V(G)$, define $\phi(v):=\sum_{u \in \v} \Phi(u)$.
So we have $B_{G^{m,s}} \Phi = B_G \phi$, and get the necessity.

On the other hand, if $ B_G x = \frac{t}{c(G)} \mathbf{1}$ has a solution $\phi$ over $\Z_m$.
Define a map $\Psi: V(G^{m,s}) \to [m]$ such that
$$ \sum_{u \in \v} \Phi(u)= \phi(v), \textrm{~for each~} v \in V(G).$$
There are $|V(G)|$ independent linear equations; such $\Phi$ is easily got
(e.g. for each $v \in V(G)$, take $\Phi(v)=\phi(v)$ and $\Phi(u)=0$ for each $u \in \v \backslash \{v\}$).
So we have $$B_{G^{m,s}} \Phi = B_G \phi=\frac{t}{c(G)} \mathbf{1}=\frac{m}{s \cdot c(G)}\mathbf{1} \textrm{~over~} \Z_{m}.$$
So $G^{m,s}$ is spectral $ s \cdot c(G)$-symmetric.
The sufficiency follows by Corollary \ref{divide}.
\end{proof}

As $G$ is spectral $c(G)$-symmetric, by Corollary \ref{sym-Zm} the equation
\begin{equation} \label{eq-ori} B_G x=\frac{t}{c(G)} \mathbf{1} \textrm{~over~} \Z_t \end{equation}
has a solution.
Obviously, if the equation (\ref{eq-power}) has a solution, then the equation (\ref{eq-ori}) has a solution as $m$ is a multiple of $t$.
However, the converse does not hold in general; see the previous counterexample.

\section{Remark}
For a nonnegative weakly irreducible tensor $\A$, its cyclic index $c(\A)$ is exactly the number of eigenvalues with modulus $\rho(\A)$.
The is implied by Perron-Frobenius theorem for nonnegative tensors,
where an eigenvalue of $\A$ is called {\it $H^+$-eigenvalue} (respectively {\it $H^{++}$-eigenvalue}) if it is associated with
a nonnegative (respectively positive) eigenvector.
For the notion of irreducible or weakly irreducible tensors, one can refer to \cite{CPZ} and \cite{FGH}.
It is known that the adjacency tensor of a uniform hypergraph $G$ is weakly irreducible if and only if $G$ is connected \cite{PZ,YY3}.

\begin{theorem}[The Perron-Frobenius Theorem for nonnegative tensors]\label{PF}~~

\begin{enumerate}

\item{\em(Yang and Yang \cite{YY3})}  If $\A$ is a nonnegative tensor, then $\rho(\A)$ is an $H^+$-eigenvalue of $\A$.

\item{\em(Friedland, Gaubert and Han \cite{FGH})} If furthermore $\A$ is weakly irreducible, then $\rho(\A)$ is the unique $H^{++}$-eigenvalue of $\A$,
with a unique positive eigenvector, up to a positive scalar.

\item{\em(Chang, Pearson and Zhang \cite{CPZ})} If moreover $\A$ is irreducible, then $\rho(\A)$ is the unique $H^{+}$-eigenvalue of $\A$,
with a unique nonnegative eigenvector, up to a positive scalar.

\end{enumerate}

\end{theorem}

According to the definition of tensor product in \cite{Shao}, for a tensor $\A$ of order $m$ and dimension $n$, and two diagonal matrices $P,Q$ both of dimension $n$,
the product $P\A Q$ has the same order and dimension as $\A$, whose entries are defined by
\[(P\A Q)_{i_1i_2\ldots i_m}=p_{i_1i_1}a_{i_1i_2\ldots i_m}q_{i_2i_2}\ldots q_{i_mi_m}.\]
If $P=Q^{-1}$, then $\A$ and $P^{m-1}\A Q$ are called {\it diagonal similar}.
It is proved that two diagonal similar tensors have the same spectrum \cite{Shao}.

\begin{theorem}[\cite{YY3}] \label{PF2}
 Let $\A$ and $\mathcal{B}$ be two $m$-th order $n$-dimensional real tensors with $|\mathcal{B}| \le  \A$,
 namely, $|b_{i_1 i_2 \ldots i_m}| \le a_{i_1 i_2 \ldots i_m}$ for each $i_j \in [n]$ and $j \in [m]$. Then

\begin{enumerate}

\item $\rho(\mathcal{B}) \le \rho(\A)$.

\item Furthermore, if $\A$ is weakly irreducible and $\rho(\mathcal{B}) = \rho(\A)$, where $\la=\rho(\A)e^{\i \theta}$ is an eigenvalue of $\mathcal{B}$ corresponding to an
eigenvector $y$,  then $y$ contains no zero entries, and $\mathcal{B}=e^{-\i \theta}D^{-(m-1)}\A D$, where $D=\hbox{diag}\{\frac{y_1}{|y_1|},\ldots,\frac{y_n}{|y_n|}\}$.
\end{enumerate}
\end{theorem}

\begin{theorem}[\cite{YY3}] \label{PF3}
Let $\A$ be an $m$-th order $n$-dimensional weakly irreducible nonnegative tensor.
Suppose $\A$ has $k$ distinct eigenvalues with modulus $\rho(\A)$ in total.
Then these eigenvalues are $\rho(\A)e^{\i \frac{2 \pi j}{k}}$, $j=0,1,\ldots,k-1$.
Furthermore, \begin{equation}\A =e^{-\i \frac{2\pi}{k}}D^{-(m-1)}\A D,\end{equation}
and the spectrum of $\A$ remains invariant under a rotation of angle $\frac{2\pi}{k}$ (but not a smaller positive angle) of the complex plane.
\end{theorem}

Suppose $\A$ be as in Theorem \ref{PF3}.
If $\S(\A)$ is invariant under a rotation of angle $\theta$ of the complex plane, i.e. $\S(\A)=e^{\i \theta}\S(\A)$,
then $\rho(\A)e^{\i \theta}$ is an eigenvalue of $\A$ by Theorem \ref{PF}.
By Theorem \ref{PF3}, $\theta=\frac{2 \pi j}{k}$ for some $j \in [k]$, and hence by Theorem \ref{PF2} (and taking $\mathcal{B}=\A$),
 $\S(\A)=e^{\i \frac{2 \pi j}{k}}\S(\A)$.
So, for some positive integer $\ell$, $\ell | k$,
  \begin{equation} \S(\A)=e^{\i \frac{2\pi}{\ell}}\S(\A).\end{equation}
The number $k$ in Theorem \ref{PF3} is exactly the cyclic index of $\A$.
In addition, if $\A$ is spectral $\ell$-symmetric,
Then $\ell \mid  c(\A)$ by Theorem \ref{PF3}.

Now return to a connected $t$-uniform hypergraph $G$ and its power $G^{m,s}$, where $m=st$.
By Lemma \ref{sym-G-power}, $G^{m,s}$ is spectral $c(G)$-symmetric; and by Lemma \ref{sym-s}, $G^{m,s}$ is also spectral $s$-symmetric.
So $G^{m,s}$ has eigenvalues $$\rho(G^{m,s})e^{\i \frac{2 \pi i}{c(G)}}e^{\i \frac{2 \pi j}{s}}, \; i \in [c(G)],j \in [s].$$ 
In particular,  $\rho(G^{m,s})e^{\i \frac{2 \pi }{d}}$ is an eigenvalue of $G^{m,s}$, where $d=\frac{s \cdot c(G)}{(s, c(G))}$.
So by Theorem \ref{PF2}, $G^{m,s}$ is spectral $d$-symmetric, which is consistent with Corollary \ref{part-conj}.


\begin{thebibliography}{99}



\bibitem{CPZ} K. C. Chang, K. Pearson, T. Zhang, \textit{Perron-Frobenius theorem for nonnegative tensors}, Commu. Math. Sci., \textbf{6} (2008), 507-520.

\bibitem{CPZ2} K. C. Chang, K. Pearson, T. Zhang, \textit{Primitivity, the convergence of the NQZ method, and the largest eigenvalue for nonnegative tensors}, SIAM J. Matrix Anal.  Appl., \textbf{32} (2009), 806-819.

\bibitem{CPZ3} K. C. Chang, K. Pearson, T. Zhang, \textit{On eigenvalue problems of real symmetric tensors},
 J. Math. Anal. Appl., \textbf{350} (2009), 416-422.

\bibitem{cooper} J. Cooper, A. Dutle, \textit{Spectra of uniform hypergraph}, Linear Algebra Appl., \textbf{436}(2012), 3268-3292.



\bibitem{FHBZL} Y.-Z. Fan, T. Huang, Y.-H. Bao, C.-L. Zhuan-Sun, Y.-P, Li,
\textit{The spectral symmetry of weakly irreducible nonnegative tensors and connected hypergraphs}, Trans. Amer. Math. Soc., DOI: https://doi.org/10.1090/tran/7741.

\bibitem{FBH} Y.-Z. Fan, Y.-H. Bao, T. Huang, \textit{Eigenvariety of nonnegative symmetric weakly irreducible tensors associated with spectral radius and its application to hypergraphs},  Linear Algebra Appl., \textbf{564} (2019), 72-94.


\bibitem{FGH} S. Friedland, S. Gaubert, L. Han, \textit{Perron-Frobenius theorem for nonnegative multilinear forms and extensions},
 Linear Algebra Appl., \textbf{438} (2013), 738-749.

\bibitem{Ha} R. Hartshorne, \textit{Algebraic Geometry}, Springer-Verlag, New York, 1977.

\bibitem{HQS13} S. Hu, L. Qi, J.-Y. Shao, \textit{Cored hypergraphs, power hypergraphs and their Laplacian H-eigenvalues}, Linear Algebra Appl., \textbf{439} (2013) 2980每2998.

\bibitem{KLQY16} L. Kang, L. Liu, L. Qi, X. Yuan, \textit{Spectral radii of two kinds of uniform hypergraphs}, Appl. Math. Comput., \textbf{338} (2018) 661-668.

\bibitem{KF15}M. Khan, Y.-Z. Fan, \textit{On the spectral radius of a class of non-odd-bipartite even uniform hypergraphs}, Linear Algebra Appl., \textbf{480} (2015) 93每106

\bibitem{KFT16} M. Khan, Y.-Z. Fan, \textit{The H-spectra of a class of generalized power hypergraphs}, Discrete Math., \textbf{339} (2016) 1682每1689.

\bibitem{Lim} L.-H. Lim, \textit{Singular values and eigenvalues of tensors: A variational approach}, in Computational Advances in Multi-Sensor Adapative Processing,
2005 1st IEEE International Workshop, IEEE, Piscataway, NJ, 2005, pp. 129-132.

\bibitem{Ni} V. Nikiforov, \textit{Hypergraphs and hypermatrices with symmetric spectrum},  Linear Algebra Appl.,  \textbf{519} (2017), 1-18.


\bibitem{PZ} K. Pearson, T. Zhang, \textit{On spectral hypergraph theory of the adjacency tensor},  Graphs Combin., \textbf{30}(5) (2014), 1233-1248.

\bibitem{Peng} X. Peng, \textit{The Ramsey number of generalized loose paths in uniform hypergraphs}, Discrete Math., \textbf{339} (2016) 539每546.

\bibitem{Qi} L. Qi, \textit{Eigenvalues of a real supersymmetric tensor}, J. Symbolic Comput., \textbf{40} (2005), 1302-1324.

\bibitem{Shao} J.-Y. Shao, \textit{A general product of tensors with applications},  Linear Algebra Appl., \textbf{439} (2013), 2350-2366.

\bibitem{SQH} J.-Y. Shao, L. Qi, S. Hu, \textit{Some new trace formulas of tensors with applications in spectral hypergraph theory},
 Linear Multilinear Algebra, \textbf{63}(5) 2015, 971-992.

\bibitem{YY} Y. Yang, Q. Yang, \textit{Further results for Perron-Frobenius theorem for nonnegative tensors}, SIAM J. Matrix Anal. Appl., \textbf{31}(5) (2010), 2517-2530.

\bibitem{YY2} Y. Yang, Q. Yang, \textit{Further results for Perron-Frobenius theorem for nonnegative tensors II},  SIAM J. Matrix Anal. Appl., \textbf{32}(4) (2011), 1236-1250.

\bibitem{YY3} Y. Yang, Q. Yang, \textit{On some properties of nonnegative weakly irreducible tensors}, arXiv: 1111.0713v2.

\bibitem{YQS16} X. Y. Yuan, L. Qi, J.-Y. Shao, \textit{The proof of a conjecture on largest Laplacian and signless Laplacian H-eigenvalues of uniform hypergraphs},
Linear Algebra Appl., \textbf{490} (2016), 18-30.

\bibitem{ZSWB14} J. Zhou, L. Sun, W. Wang, C. Bu, \textit{Some spectral properties of uniform hypergraphs}, Elect. J. Combin., \textbf{21}(4) (2014), \#P4.24, 14.


\end{thebibliography}
\end{document}